\documentclass[preprint, 11pt, english]{elsarticle}
\usepackage{amsmath}
\usepackage{latexsym, amssymb, mathtools}
\usepackage{amsthm}
\usepackage{color}
\usepackage{txfonts}
\usepackage{soul}

\newtheorem{thm}{Theorem}[section] 

\newtheorem{cor}[thm]{Corollary}
\newtheorem{defn}[thm]{Definition}

\newtheorem{prop}[thm]{Proposition}

\theoremstyle{definition}
\newtheorem{rem}[thm]{Remark}

\newcommand\operA[2]{{\if!#2!\operatorname{#1}\else{\operatorname{#1}_{#2}^{\phantom{I}}}\fi}} 

\newcommand\Pref[1]{{Proposition~\ref{#1}}}%
\newcommand\Cref[1]{{Corollary~\ref{#1}}}%
\newcommand\Rref[1]{{Remark~\ref{#1}}}%
\newcommand\Tref[1]{{Theorem~\ref{#1}}}%

\def\tr{{\operatorname{Tr}}}

\def\dim{{\operatorname{dim}}}

\def\norm{{\operatorname{N}}}

\newcommand{\Trace}[1][]{\if!#1!\operatorname{Tr}\else{\operatorname{Tr}_{#1}^{\phantom{I}}}\fi} 

\long\def\forget#1\forgotten{{}} %

\def\({\left(}
\def\){\right)}

\usepackage{tabu}

\newif\iffurther
\furtherfalse

\newif\ifXY 
\XYtrue     
\ifXY

\input xy
\input xyidioms.tex
\usepackage{xy}
\xyoption{all} %
\fi 

\usepackage{longtable}
\usepackage{babel}

\journal{Archiv der Mathematik}

\begin{document}

\begin{frontmatter}

\title{Symbol $p$-Algebras of Prime Degree and their $p$-Central Subspaces}

\author{Adam Chapman}
\ead{adam1chapman@yahoo.com}
\address{Department of Computer Science, Tel-Hai College, Upper Galilee, 12208 Israel}
\author{Michael Chapman}
\ead{michael169chapman@gmail.com}

\address{Department of Mathematics, Ben-Gurion University of the Negev,
	P.O. Box 653, Be’er-Sheva 84105, Israel}

\begin{abstract}
We prove that the maximal dimension of a $p$-central subspace of the generic symbol $p$-algebra of prime degree $p$ is $p+1$. We do it by proving the following number theoretic fact: let $\{s_1,\dots,s_{p+1}\}$ be $p+1$ distinct nonzero elements in the additive group $G=(\mathbb{Z}/p \mathbb{Z}) \times (\mathbb{Z}/p \mathbb{Z})$; then every nonzero element $g \in G$ can be expressed as $d_1 s_1+\dots+d_{p+1} s_{p+1}$ for some non-negative integers $d_1,\dots,d_{p+1}$ with $d_1+\dots+d_{p+1} \leq p-1$.
\end{abstract}


\end{frontmatter}

\section{Introduction}

Let $p$ be a prime integer and let $F$ be a field.
We study symbol $p$-algebras of degree $p$, i.e. central simple algebras of degree $p$ over $F$ with $\operatorname{char}(F)=p$.
Such a symbol algebra is of the form
$$A=F \langle x,y : x^p-x=\alpha, y^p=\beta, y x y^{-1}=x+1 \rangle$$
for some $\alpha \in F$ and $\beta \in F^\times$.
We denote this algebra by the symbol $[\alpha,\beta)_{p,F}$.
It is a division algebra if and only if $F[x : x^p-x=\alpha]$ is a field extension of $F$ and $\beta$ is not a norm in this field extension. Otherwise it is isomorphic to the $p \times p$ matrix algebra $M_p(F)$ over $F$.
The $p$-torsion of $Br(F)$ is generated by such algebras (proven originally by Teichm\"uller, see \cite[Theorem 9.1.4]{GS} and \citep[Chapter 7, Theorem 30]{Albert}).
The fact that the $p$-torsion of $Br(F)$ is generated by symbol algebras in the case of $\operatorname{char}(F) \neq p$ and $F$ containing primitive $p$th roots of unity was proven only a few decades later in \cite{MS}.

An element $z \in A$ is called $p$-central if $z^p \in F$.
If $z$ is $p$-central and not central then one can write $A$ as $[\alpha,z^p)_{p,F}$ for some $\alpha \in F$.
These elements are therefore vital for understanding the structure of $A$ and the different symbol presentations it can take.
\begin{defn}
	An $F$-vector subspace of $A=[\alpha,\beta)_{p,F}$ consisting only of $p$-central elements is called a $p$-central subspace of $A$.
\end{defn}
A key example of a $p$-central subspace of $A$ is $F[x]y=F y+F xy+\dots+F x^{p-1} y$.
For any nonzero $z=f(x) y \in F[x]y$, one can write $A=[\alpha,z^p)_{p,F}=[\alpha,\operatorname{N}_{F[x]/F}(f(x)) \beta)_{p,F}$ (see \citep[Chapter 7, Lemma 10]{Albert}).
This symbol modification explains why $\beta$ must not be a norm in order for the algebra to be a division algebra:
if $\beta$ is the norm of some $f(x)$ then for $z=f(x)^{-1} y$ we get $A=[\alpha,\operatorname{N}_{F[x]/F}(f(x)^{-1}) \beta)_{p,F}=[\alpha,1)_{p,F}$ which contains a nilpotent element and thus is clearly not a division algebra.
This treatment of $p$-central spaces was extended in \cite{Chapman:sub} to tensor products of symbol algebras in order to bound the symbol length of algebras of exponent $p$ over fields with a prescribed upper bound on the dimension of anisotropic polynomial forms of degree $p$, following the example of \cite{Matzri} that treated such spaces in the case of $\operatorname{char}(F) \neq p$ and $F$ containing primitive $p$th roots of unity. 

We are interested in the $p$-central subspaces of $A$ and above all in their maximal dimension.
We conjecture that the maximal dimension is $p+1$, noting that one can extend the key example mentioned above to the $(p+1)$-dimensional $p$-central space $F[x] y+F$.
This is known to be true when $p=2$ or 3: for $p=2$ it is enough to notice that the subspace of elements of trace zero is 3-dimensional; for $p=3$ see \cite[Theorem 6.1]{MV2}.

In this paper, we prove the conjecture in the ``generic case", i.e. for a symbol algebra $[\alpha,\beta)_{p,F}$ where $F$ is either the function field $K(\alpha,\beta)$ in two algebraically independent variables $\alpha$ and $\beta$ or the field $K((\alpha^{-1}))((\beta^{-1}))$ of iterated Laurent series over some field $K$ with $\operatorname{char}(K)=p$.
An equivalent statement was proven in the case of $\operatorname{char}(F) \neq p$ and $F$ containing primitive $p$th roots of unity in \cite{CGMRV}.
We prove the main statement by reducing the problem into a number theoretic question and answering this question independently.

\section{Preliminaries}

\subsection{The trace and norm forms}
Let $p$ be a prime integer and let $F$ be a field with $\operatorname{char}(F)=p$.
Let $A=[\alpha,\beta)_{p,F}=F \langle x,y : x^p-x=\alpha, y^p=\beta, y x y^{-1}=x+1 \rangle$ be a symbol $p$-algebra of degree $p$ over $F$.
For any maximal subfield $E$ of $A$, the algebra $A \otimes E$ is isomorphic to $M_p(E)$.
There is therefore a natural embedding of $\Phi: A \hookrightarrow M_p(E)$.
The trace and determinant of any element in $\Phi(A)$ are in $F$ (see \cite[Section 2.6]{GS}).
We can therefore consider the trace form $\tr : A \rightarrow F$ mapping each $\lambda \in A$ to $\tr(\Phi(\lambda))$, and the norm form $\norm : A \rightarrow F$ mapping each $\lambda$ to $\det(\Phi(\lambda))$.
In particular, the identity element $1$ in $F$ is mapped to the identity matrix in $M_p(E)$ whose trace is $p$, i.e. $0$.
Note that $\norm(zt)=\norm(z) \norm(t)$, $\tr(z+t)=\tr(z)+\tr(t)$ and $\tr(c z)=c \tr(z)$ for any $z,t \in A$ and $c \in F$.

Another way to understand the trace form is the following:
every noncentral element $\lambda$ in $[\alpha,\beta)_{p,F}$ generates a field extension of degree $p$ over $F$.
Therefore it satisfies some minimal polynomial equation
$$\lambda^p+c_{p-1} \lambda^{p-1}+\dots+c_1 \lambda+c_0=0.$$
The trace $\tr(\lambda)$ of $\lambda$ is $-c_{p-1}$ and the norm $\norm(\lambda)$ of $\lambda$ is $-c_0$.
Specifically, for any $\lambda$ in $F[x]$, $\tr(\lambda)=\lambda+\sigma(\lambda)+\dots+\sigma^{p-1}(\lambda)$ and $\norm(\lambda)=\lambda \sigma(\lambda) \dots \sigma^{p-1}(\lambda)$ where $\sigma$ is the automorphism of $F[x]$ fixing $F$ and mapping $x$ to $x+1$.
Note that $\sigma(x)=y x y^{-1}$ and $\norm(x)=\alpha$.

Every element $z$ in $A$ can be written as $\sum_{i=0}^{p-1} \sum_{j=0}^{p-1} a_{i,j} x^i y^j$ for some $a_{i,j} \in F$.
In order to compute the trace of $z$, it is therefore enough to know the trace of each $x^i y^j$.
If $j \neq 0$ then $(x^i y^j)^p=x^i \sigma^j (x^i) \dots \sigma^{(p-1) j}(x^i) (y^j)^p =\norm(x^i) (y^p)^j=\alpha^i \beta^j$ and so $\tr(x^i y^j)=0$.

Now, for any $i \in \{0,1,\dots,p-2\}$, we have
\[
\tr(x^i)=x^i+\sigma(x^i)+\dots+\sigma^{p-1}(x^i)
=\sum_{k=0}^{p-1} (x+k)^i=\sum_{k=0}^{p-1} \sum_{\ell=0}^i \binom{i}{\ell} k^{\ell} x^{i-\ell}.
\]

\begin{rem}\label{shek}
For each $\ell$ in $\{0,\dots,i\}$ we have $\sum_{k=0}^{p-1} k^{\ell}=0$, and so $\tr(x^i)=0$.
\end{rem}

This fact is well-known and follows directly from Newton's identities and the characteristic polynomial of $x$. We present here an alternative proof:

\begin{proof}
Note that $$\sum_{k=0}^{p-1} \sum_{\ell=0}^i \binom{i}{\ell} k^{\ell} x^{i-\ell}
=\sum_{\ell=0}^i \left(\sum_{k=0}^{p-1} k^{\ell}\right) \binom{i}{\ell} x^{i-\ell}.$$
For $\ell=0$ we have $$\sum_{k=0}^{p-1} k^{\ell}=\underbrace{1+\dots+1}_{p \  \text{times}}=0.$$
Suppose $\ell \neq 0$.
Note that the multiplicative group $(\mathbb{Z}/p \mathbb{Z})^\times$ is cyclic of order $p-1$. Let $g$ be its generator. Then 
$$\sum_{k=0}^{p-1} k^{\ell}=\sum_{k=1}^{p-1} k^{\ell}=\sum_{r=0}^{p-2} (g^r)^{\ell}=\sum_{r=0}^{p-2} (g^\ell)^r=\frac{(g^{\ell})^{p-1}-1}{g^{\ell}-1}.$$ Since $1 \leq \ell \leq i \leq p-2$, $g^{\ell} \neq 1$ whereas $(g^{\ell})^{p-1}=1$. 
Hence $$\frac{(g^{\ell})^{p-1}-1}{g^{\ell}-1}=\frac{0}{g^{\ell}-1}=0.\qedhere$$ 
\end{proof}

From the equality $x^p-x=\alpha$ we get $(x^{-1})^p+\frac{1}{\alpha} (x^{-1})^{p-1}-\frac{1}{\alpha}$, which means $\tr(x^{-1})=-\frac{1}{\alpha}$. Similarly, $x^{p-1}=1+\alpha x^{-1}$, and so $\tr(x^{p-1})=\tr(1)+\alpha \tr(x^{-1})=-1$.
We can also derive this fact as a corollary of \Rref{shek} in the following way:
\[
\tr(x^{p-1})=\sum_{k=0}^{p-1} \sum_{\ell=0}^{p-1} \binom{p-1}{\ell} k^{\ell} x^{p-1-\ell}=\sum_{k=0}^{p-1}k^{p-1},
\]
and by Fermat's little theorem,
\[
\sum_{k=0}^{p-1} k^{p-1}=0+\underbrace{1+\dots+1}_{p-1 \  \text{times}}=p-1=-1.
\]

We outline these computations in the following remark:

\begin{rem} \label{remushim}
	The trace form $\tr : A \rightarrow F$ maps every element $ \sum_{i=0}^{p-1} \sum_{j=0}^{p-1} a_{i,j} x^i y^j$ to $-a_{p-1,0}$.
\end{rem}

\subsection{Trace condition for being $p$-central}

Let $v_1,...,v_m$ be elements of $A$ and $d_1,...,d_m$ be non-negative integers. The notation $v_1^{d_1} * \dots * v_{m}^{d_{m}}$ stands for the sum of all the possible products of $d_1$ copies of $v_1$, $d_2$ copies of $v_2$ and so on (see \cite[\S1.2]{Revoy}). For example, $v_1^2 * v_2=v_1^2 v_2+v_1 v_2 v_1+v_2 v_1^2$.

Consider the $F$-vector subspace $V=F v_1+\dots+F v_m$ of $A$.
A necessary and sufficient condition for $V$ to be $p$-central is $\tr(v_1^{d_1} * \dots * v_m^{d_m})=0$ for every choice of non-negative integers $d_1,\dots,d_m$ satisfying $d_1+\dots+d_m \leq p-1$ (see \cite[Theorem 36]{MRSV}).
Note that although in this condition we are using a specific basis of $V$, the property of being $p$-central is independent of the choice of basis.

\begin{rem}\label{extension}
Let $L$ be some field extension of $F$ and $B=A \otimes L$.
Let $W=L v_1+\dots+L v_m$ the scalar extension of $V$ from $F$ to $L$. Then by the necessary and sufficient condition for being $p$-central mentioned above, if $V$ is $p$-central in $A$ then $W$ is $p$-central in $B$.
\end{rem}

\section{Maximal $p$-Central Subspaces in the Generic Algebra}

\begin{thm} \label{mainthm}
Let $p$ be a prime number, $K$ be a field with $\operatorname{char}(K)=p$ and $F$ be either the function field $K(\alpha,\beta)$ in two algebraically independent variables over $K$ or the field of iterated Laurent series $K((\alpha^{-1}))((\beta^{-1}))$. Then the maximal dimension of a $p$-central subspace of $[\alpha,\beta)_{p,F}$ is $p+1$.
\end{thm}

\sloppy
The rest of the paper is dedicated to proving this theorem.
By Remark \ref{extension}, every $p$-central subspace of $[\alpha,\beta)_{p,K(\alpha,\beta)}$ gives rise to a $p$-central subspace of $[\alpha,\beta)_{p,K((\alpha^{-1}))((\beta^{-1}))}$ of the same dimension. Therefore it is enough to prove the theorem for $K((\alpha^{-1}))((\beta^{-1}))$.
Moreover, in \S1 we gave an example of a $p$-central subspace of dimension $p+1$. Hence, it is enough to show that every $(p+2)$-dimensional subspace of $A$ is not $p$-central.

Let $F=K((\alpha^{-1}))((\beta^{-1}))$,
$A=F \langle x,y : x^p-x=\alpha, y^p=\beta, y x y^{-1}=x+1 \rangle=[\alpha,\beta)_{p,F}$,
and $\frak{v}$ be the right-to-left $(\alpha^{-1},\beta^{-1})$-adic Henselian valuation on $F$. Recall that the value group $\Gamma_F$ of $F$ is $\mathbb{Z} \times \mathbb{Z}$. For general introduction to valuation theory on division algebras see \cite{TignolWadsworth:2015}.

\begin{rem}
The algebra $A$ is a division algebra.
\end{rem}

\begin{proof}
We use the necessary and sufficient condition for a symbol algebra to be a division algebra mendtioned in \S1.  
Consider the equation $\lambda^p-\lambda=\alpha$ over $F$.
Suppose it has a root $z$.
Then $z+k$ is also a root for any $k \in \mathbb{Z}/p \mathbb{Z}$.
If $\frak{v}(z) \geq (0,0)$ then $z^p-z=z (z+1) \cdot ... \cdot (z+p-1)$ must have a nonnegative value.
However, $\frak{v}(\alpha)=(-1,0)$, which means that $\frak{v}(z)<(0,0)$.
Therefore $\frak{v}(z+k)=\frak{v}(z)$ for any $k \in \mathbb{Z}/p \mathbb{Z}$, and so $\frak{v}(z)=\frac{1}{p} \frak{v}(\alpha)=(-\frac{1}{p},0)$ which is not in $\Gamma_F$, contradiction.
Hence $F[x: x^p-x=\alpha]$ is a field.
Its value group is $\frac{1}{p} \mathbb{Z} \times \mathbb{Z}$.
Every norm in the field extension $F[x : x^p-x=\alpha]/F$ has a value in $\mathbb{Z} \times p \mathbb{Z}$. 
Since $v(\beta)=(0,-1)$, $\beta$ cannot be a norm in this field extension.
\end{proof}

Let $V=F v_1+\dots+F v_{p+2}$ be a $(p+2)$-dimensional subspace of $A$. We are going to prove that $V$ is not $p$-central.
Since $A$ is a division algebra and $\frak{v}$ is Henselian, the valuation $\frak{v}$ extends uniquely to $A$ (\cite[Theorem 1.4]{TignolWadsworth:2015}).
Note that $\frak{v}(x)=(-\frac{1}{p},0)$ and $\frak{v}(y)=(0,-\frac{1}{p})$. Thus $\Gamma_A=\frac{1}{p} \mathbb{Z} \times \frac{1}{p} \mathbb{Z}$ and $\Gamma_A/\Gamma_F\cong \mathbb{Z}/p\mathbb{Z}\times \mathbb{Z}/p\mathbb{Z}$. Since $\dim A=p^2=|\Gamma_A/\Gamma_F|$, $A$ is totally ramified.
Let $\varphi\colon \Gamma_A\rightarrow \Gamma_A/\Gamma_F$ be the quotient map.
By \cite[Proposition 3.14]{TignolWadsworth:2015}, we have $|\varphi(\Gamma_V)|=[V:F]$. 
Thus we can choose an $F$-basis $v_1,\dots,v_{p+2}$ for $V$ whose values are distinct elements in $\{0,-\frac{1}{p},\dots,-\frac{p-1}{p}\} \times \{0,-\frac{1}{p},\dots,-\frac{p-1}{p}\}$ (see also \cite[Remark 2.2]{ChapmanUre}). For every $k \in \{1,\dots,p+2\}$ let $(i_k,j_k)$ be $-p\frak{v}(v_k)$.

\begin{prop} \label{propprop}
Suppose there are non-negative integers $d_1,\dots,d_{p+2}$ with $d_1+\dots+d_{p+2} \leq p-1$ such that $d_1 i_1\dots+d_{p+2} i_{p+2} \equiv p-1 \pmod{p}$ and $d_1 j_1+\dots+d_{p+2} j_{p+2} \equiv 0 \pmod{p}$.
Then $\tr(v_1^{d_1} * \dots * v_{p+2}^{d_{p+2}}) \neq 0$ and so $V$ is not $p$-central.
\end{prop}

\begin{proof}
\sloppy Recall that each element $z \in A$ can be written uniquely as $z=\sum_{i=0}^{p-1} \sum_{j=0}^{p-1} a_{i,j} x^i y^j$ where $a_{i,j} \in F$ for any $i,j \in \{0,\dots,p-1\}$.
All the nonzero terms in this sum have distinct values, because they are distinct modulo $\Gamma_F=\mathbb{Z} \times \mathbb{Z}$.
There is therefore one term $a_{i_0,j_0} x^{i_0} y^{j_0}$ of minimal value which determines the value of $z$.
The coefficient $a_{i_0,j_0}$ is a Laurent series in $K((\alpha^{-1}))((\beta^{-1}))$, so it also has a term of minimal value $c \alpha^{r_0} \beta^{s_0}$ for some $r_0,s_0 \in \mathbb{Z}$ and nonzero $c \in K$.
Let $\widetilde{z}$ denote $c \alpha^{r_0} \beta^{s_0} x^{i_0} y^{j_0}$. Note that $\frak{v}(z)=-\frac{1}{p}(pr_o+i_0,ps_0+j_0)$, so the value of $z$ determines $\widetilde{z}$ up to a nonzero scalar from $K$. Since we can multiply the basis elements by scalars from $F$, we may assume $\widetilde{v_k}=x^{i_k} y^{j_k}$ for each $k \in \{1,\dots,p+2\}$.

Since $yx=xy+y$ and $x^p=\alpha+x$ where $\frak{v}(y)>\frak{v}(xy)$ and $\frak{v}(x)>\frak{v}(\alpha)$, for any $r_0,s_0,r_1,s_2 \in \mathbb{Z}$ and $i_0,j_0,i_1,j_1 \in \{0,\dots,p-1\}$ we have $(\alpha^{r_0} \beta^{s_0} x^{i_0} y^{j_0})(\alpha^{r_1} \beta^{s_1} x^{i_1} y^{j_1})=\alpha^{r_2} \beta^{s_2} x^{i_2} y^{j_2}+S$ where $i_2$ and $j_2$ are the unique integers in $\{0,\dots,p-1\}$ with $i_2 \equiv  i_0+i_1 \pmod{p}$ and $j_2 \equiv j_0+j_1 \pmod{p}$, $r_2=r_0+r_1+\frac{i_0+i_1-i_2}{p}$, $s_2=s_0+s_1+\frac{j_0+j_1-j_2}{p}$, and $\frak{v}(S)>\frak{v}(\alpha^{r_2} \beta^{s_2} x^{i_2} y^{j_2})=-\frac{1}{p}(pr_0+pr_1+i_0+i_1,ps_0+ps_1+j_0+j_1)$.
Consequently, if $\widetilde{z_0}=\alpha^{r_0} \beta^{s_0} x^{i_0} y^{j_0}$ and $\widetilde{z_1}=\alpha^{r_1} \beta^{s_1} x^{i_1} y^{j_1}$ then $\widetilde{z_0z_1}=\alpha^{r_2} \beta^{s_2} x^{i_2} y^{j_2}$.

Recall that $\Sigma=v_1^{d_1} * \dots * v_{p+2}^{d_{p+2}}$ is the sum of products of $d_1$ copies of $v_1$, $d_2$ copies of $v_2$ etc.
For each summand $\pi$ in $\Sigma$,
$$\frak{v}(\pi)=-\frac{1}{p}\left(d_1 (i_1,j_1)+\dots+d_{p+2} (i_{p+2},j_{p+2})\right).$$

Since $d_1 i_1+\dots+d_{p+2} i_{p+2} \equiv p-1 \pmod{p}$ and $d_1 j_1+\dots+d_{p+2} j_{p+2} \equiv 0 \pmod{p}$, we have $\widetilde{\pi}=\alpha^r \beta^s x^{p-1}$ where
\[
r=\frac{d_1 i_1+\dots+d_{p+2} i_{p+2}-p+1}{p} \enspace \operatorname{and} \enspace s=\frac{d_1 j_1+\dots+d_{p+2} j_{p+2}}{p}.
\]
Notice that $n=\binom{d_1+\dots+d_{p+2}}{d_1,\dots,d_{p+2}}$ is the number of terms in $\Sigma$. 
Since $d_1+\dots+d_{p+2} \leq p-1$ and $p$ is prime, $n$ is not a multiple of $p$. 
Therefore $\widetilde{\Sigma}$ is $n \alpha^r \beta^s x^{p-1}$, and by \Rref{remushim} the trace of $\Sigma$ is a Laurent series whose leading term is $-n \alpha^{r} \beta^s$, and thus it is nonzero.
\end{proof}

In the following section we prove that the conditions of \Pref{propprop} are satisfied.

\section{The Number Theoretic Problem}

\begin{thm}\label{numbertheorem}
Let $p$ be a prime integer, $G$ be the group $\mathbb{Z}/p \mathbb{Z} \times \mathbb{Z}/p \mathbb{Z}$ and $S=\{s_1,\dots,s_{p+1}\}$ be $p+1$ distinct nonzero elements of $G$. 
Then for any nonzero $g$ in $G$, there exist non-negative integers $d_1,\dots,d_{p+1}$ with $\sum_{i=1}^{p+1} d_i \leq p-1$ such that $d_1 s_1+\dots+d_{p+1} s_{p+1}=g$.
\end{thm}

The set $\{(i_1,j_1),...,(i_{p+2},j_{p+2})\}$ from \Pref{propprop} consists of $p+2$ distinct elements in $G$. Thus there are at least $p+1$ nonzero elements in this set.
If we take $g$ to be $(p-1,0)$ and $S$ to be $p+1$  nonzero elements from $\{(i_1,j_1),...,(i_{p+2},j_{p+2})\}$, then the conditions of \Pref{propprop} are satisfied. 
Thus by proving \Tref{numbertheorem}, we complete the proof of \Tref{mainthm}.

\begin{prop}\label{usefulprop}
Suppose $p$ is an odd prime and $n$ a positive integer.
Let $a_1,\dots,a_n$ be integers prime to $p$ with $a_1+\dots+a_n \not \equiv 1 \pmod{p}$.
Then for any integers $b_1,\dots,b_n$ there exist non-negative integers $d_1,\dots,d_{n+1}$ with $d_1+\dots+d_{n+1} \leq \frac{n}{2}(p-1)$ such that $d_k+d_{n+1} a_k \equiv b_k \pmod{p}$ for every $k \in \{1,\dots,n\}$.
\end{prop}

\begin{rem}\label{permutation}
For any integers $a$ and $b$ with $\gcd(a,p)=1$, the function
$\sigma : \{0,1,\dots,p-1\} \rightarrow \{0,1,\dots,p-1\}$ mapping each $t$ to the representative of the $\pmod{p}$-congruence class of $b-at$ is  injective, and so $\sigma$ is a permutation.
\end{rem}

\begin{proof}[Proof of \Pref{usefulprop}]
Since $a_1,\dots,a_n$ are prime to $p$, Remark \ref{permutation} implies that there are permutations
$\sigma_1,\dots,\sigma_n : \{0,\dots,p-1\} \rightarrow \{0,\dots,p-1\}$ satisfying $\sigma_k(t)+t a_k \equiv b_k \pmod{p}$ for any $k \in \{1,\dots,n\}$ and $t \in \{0,\dots,p-1\}$.
Let $\varphi : \{0,\dots,p-1\} \rightarrow \mathbb{Z}$ be the function defined by $\varphi(t)=t+\sigma_1(t)+\dots+\sigma_n(t)$.
Since 
\[
\varphi(t) \equiv \sum_{k=1}^n b_k+t \(1-\sum_{k=1}^n a_k\) \pmod{p}
\]
and
\[ 
 1-\sum_{k=1}^n a_k \not \equiv 0 \pmod{p},
\]
the integers $\varphi(0),\dots,\varphi(p-1)$ belong to different $\pmod{p}$-congruence classes, and so they are different in pairs as integers.

Now 
\[
\begin{split}
\sum_{t=0}^{p-1}\varphi(t)&=\sum_{t=0}^{p-1} \(t+\sigma_1(t)+\dots+\sigma_n(t)\)\\
&=\sum_{t=0}^{p-1} t+\sum_{t=0}^{p-1}\sigma_1(t)+\dots+\sum_{t=0}^{p-1}\sigma_n(t)\\
&=\frac{(n+1) p (p-1)}{2}.
\end{split}
\]
If $\varphi(t) \geq \frac{n (p-1)}{2}+1$ for each $t \in \{0,\dots,p-1\}$,
then since $\varphi(0),...,\varphi(p-1)$ are distinct integers we have
\[
\sum_{t=0}^{p-1}\varphi(t) \geq \sum_{i=1}^p\(\frac{n (p-1)}{2}+i\)=\frac{np(p-1)}{2} + \frac{p(p+1)}{2}>\frac{(n+1) p (p-1)}{2},
\] 
contradiction.

Consequently there exists some $t \in \{0,\dots,p-1\}$ for which $\varphi(t) \leq \frac{n (p-1)}{2}$.
Take then $d_{n+1}=t$ and $d_k=\sigma_k(t)$ for any $k \in \{1,\dots,n\}$.
\end{proof}

\begin{cor} \label{usefulcor}
	Let $p$ be an odd prime and $G=\mathbb{Z}/p\mathbb{Z}\times \mathbb{Z}/p\mathbb{Z}$.
	\begin{enumerate}
	\item[(1)] Let $s_1,s_2,s_3$ in $G$ be linearly independent in pairs where $s_3=as_1+bs_2$ and $a+b\not \equiv 1\pmod{p}$. Then for every nonzero element $g$ of $G$ there exist non-negative integers $d_1,d_2,d_3$ where $d_1+d_2+d_3\leq p-1$ such that $g=d_1s_1+d_2s_2+d_3s_3$.
	\item[(2)] Let $s_1,s_2,s_3,s_4$ be different nonzero elements of $G$ where $s_2\in\langle s_1 \rangle$, $s_4\in\langle s_3\rangle$ and $\langle s_1 \rangle \cap \langle s_3 \rangle = \{(0,0)\}$. Then for every nonzero element $g$ of $G$ there exist non-negative integers $d_1,d_2,d_3,d_4$ where $d_1+d_2+d_3+d_4\leq p-1$ such that $g=d_1s_1+d_2s_2+d_3s_3+d_4s_4$.
	\end{enumerate}
\end{cor}

\begin{proof}
	\item[(1)] Since $s_1,s_2$ are linearly independent, $G=\langle s_1,s_2\rangle$ and we can present $g$ as $e_1s_1+e_2s_2$.  Taking in \Pref{usefulprop} 
	\[
	n=2, a_1=a, a_2=b, b_1=e_1, b_2=e_2 
	\]
	we get non-negative integers $d_1,d_2,d_3$ where $d_1+d_3+d_3\leq p-1$ such that 
	\[
	d_1+d_3a\equiv e_1 \pmod{p};\quad d_2+d_3b\equiv e_2\pmod{p}.
	\] 
	Therefore $d_1s_1+d_2s_2+d_3s_3=g$ and $d_1+d_2+d_3\leq p-1$.
	
	\item[(2)] Since $\langle s_1\rangle\cap\langle s_3 \rangle=\{(0,0)\}$ we have $G=\langle s_1,s_3 \rangle$ and can present $g$ as $e_1s_1+e_3s_3$ for some $e_1,e_3\in\{0,...,p-1\}$.
	 Moreover $s_2=as_1$ and $s_4=bs_3$ for some $a,b\not \equiv 0,1
	 \pmod{p}$.
	 If $e_1=0$ or $e_3=0$, then we can present $g$ as $e_3s_3$ or $e_1s_1$ and clearly $e_1,e_3\leq p-1$.
	 Otherwise we use \Pref{usefulprop} twice: once with $	n=1, a_1=a, b_1=e_1$,
	and the second time with $	n=1, a_1=b, b_1=e_3$.
	Thus we get $d_1,d_2,d_3,d_4$ where $d_1+d_2\leq \frac{p-1}{2},d_3+d_4\leq \frac{p-1}{2}$ such that
	\[
	d_1+d_2a\equiv e_1\pmod{p};\quad d_3+d_4b \equiv e_2\pmod{p}. 
	\]
	Therefore $d_1s_1+d_2s_2+d_3s_3+d_4s_4=g$ and $d_1+d_2+d_3+d_4\leq p-1$.
\end{proof}

We are now ready to prove the main theorem of this section.

\begin{proof}[Proof of Theorem \ref{numbertheorem}]
If $p=2$, then since $G$ has exactly $3$ nonzero elements, $S=G\setminus \{(0,0)\}$ and $g\in S$.

Let $p$ be an odd prime.
The number of proper nonzero subgroups of $G$ is $p+1$, and each one contains $p-1$ nonzero elements.
Thus, by the pigeonhole principle, there are two cases to deal with:
\begin{enumerate}
	\item[(1)] The set $S$ intersect only two of the proper nonzero subgroups of $G$.
	\item[(2)] The set $S$ intersects at least three of the proper nonzero subgroups of $G$;
\end{enumerate} 

\textit{Case (1)} - In this case, again due to the pigeonhole principle, in each one of the two proper subgroups there are at least two elements of $S$, say $s_1,s_2\in\langle s_1 \rangle$, $s_3,s_4 \in\langle s_3\rangle$ and $\langle s_1 \rangle \cap \langle s_3 \rangle = \{(0,0)\}$. Thus by \Cref{usefulcor}(2) we are done. 

\textit{Case (2)} - This case splits into two subcases:
\begin{enumerate}
	\item[(a)] Each element of $S$ is in a different proper nonzero subgroup.
	\item[(b)] Two of the elements of $S$ are in the same proper nonzero subgorup.
\end{enumerate}
In Case (a), for $s_1$ and $s_2$ we have $G=\langle s_1,s_2 \rangle$. 
Thus all other elements of $S$ can be presented as $s_i=a_is_1+b_is_2$. 
Since there are $p$ elements in $G$ of the form $as_1+bs_2$ with $a+b\equiv 1\pmod{p}$, and there are $p+1$ elements in $S$, by the pigeonhole principle, one of them  must satisfy $a_i+b_i\not \equiv 1 \pmod{p}$, say $s_3$. Therefore by using \Cref{usefulcor}(1) with $s_1,s_2,s_3$ we are done. 

In Case (b), say $s_3,s_4$ are from the same proper nonzero subgroup of $G$ and $s_1,s_2$ are each from one of the other two proper nonzero subgroups that $S$ intersects. 
Then $s_3=a_3s_1+b_3s_2$ and $s_4=ms_3$ for some integer $m\not \equiv 0,1 \pmod{p}$. Thus, given $a_3+b_3 \not \equiv 1 \pmod{p}$ we use the triplet $s_1,s_2,s_3$ in \Cref{usefulcor}(1). 
Otherwise we conclude that $ma_3+mb_3\equiv m\not \equiv 1 \pmod{p}$ and use the triplet $s_1,s_2,s_4$ in \Cref{usefulcor}(1).
\end{proof}

\section*{Acknowledgments}

The authors thank Yotam Hendel, Jean-Pierre Tignol and the anonymous referee for their comments on the manuscript. Part of the research was carried out in the summer of 2016 during the time the first author was visiting Perimeter Institute of Theoretical Physics and the second author was visiting J\'an Min\' a\v c at the University of Western Ontario.

\section*{Bibliography}
\bibliographystyle{amsalpha}
\bibliography{bibfile}

\newcommand{\etalchar}[1]{$^{#1}$}
\def\cprime{$'$}
\providecommand{\bysame}{\leavevmode\hbox to3em{\hrulefill}\thinspace}
\providecommand{\MR}{\relax\ifhmode\unskip\space\fi MR }
\providecommand{\MRhref}[2]{%
  \href{http://www.ams.org/mathscinet-getitem?mr=#1}{#2}
}
\providecommand{\href}[2]{#2}
\begin{thebibliography}{CGM{\etalchar{+}}16}

\bibitem[Alb61]{Albert}
A.~Adrian Albert, \emph{Structure of algebras}, Revised printing. American
  Mathematical Society Colloquium Publications, Vol. XXIV, American
  Mathematical Society, Providence, R.I., 1961. \MR{0123587 (23 \#A912)}

\bibitem[CGM{\etalchar{+}}16]{CGMRV}
Adam Chapman, David~J. Grynkiewicz, Eliyahu Matzri, Louis~H. Rowen, and Uzi
  Vishne, \emph{Kummer spaces in symbol algebras of prime degree}, J. Pure
  Appl. Algebra \textbf{220} (2016), no.~10, 3363--3371. \MR{3497965}

\bibitem[Cha17]{Chapman:sub}
Adam Chapman, \emph{Symbol length of $p$-algebras of prime exponent}, J.
  Algebra Appl. \textbf{16} (2017), no.~5, 1750136, 9.

\bibitem[CU]{ChapmanUre}
Adam Chapman and Charlotte Ure, \emph{Tensor products of cyclic algebras of
  degree 4 and their kummer subspaces}, Comm. Algebra \textbf{to appear}.

\bibitem[GS06]{GS}
Philippe Gille and Tam{\'a}s Szamuely, \emph{Central simple algebras and
  {G}alois cohomology}, Cambridge Studies in Advanced Mathematics, vol. 101,
  Cambridge University Press, Cambridge, 2006. \MR{2266528}

\bibitem[Mat16]{Matzri}
Eliyahu Matzri, \emph{Symbol length in the {B}rauer group of a field}, Trans.
  Amer. Math. Soc. \textbf{368} (2016), no.~1, 413--427. \MR{3413868}

\bibitem[MRSV]{MRSV}
Eliyahu Matzri, Louis~H. Rowen, David Saltman, and Uzi Vishne, \emph{Bimodule
  structure of central simple algebras}, arXiv:1601.07570.

\bibitem[MS82]{MS}
A.~S. Merkur{\cprime}ev and A.~A. Suslin, \emph{{$K$}-cohomology of
  {S}everi-{B}rauer varieties and the norm residue homomorphism}, Izv. Akad.
  Nauk SSSR Ser. Mat. \textbf{46} (1982), no.~5, 1011--1046, 1135--1136.
  \MR{675529 (84i:12007)}

\bibitem[MV14]{MV2}
Eliyahu Matzri and Uzi Vishne, \emph{Composition algebras and cyclic
  {$p$}-algebras in characteristic 3}, Manuscripta Math. \textbf{143} (2014),
  no.~1-2, 1--18. \MR{3147442}

\bibitem[Rev77]{Revoy}
Ph. Revoy, \emph{Alg\`ebres de {C}lifford et alg\`ebres ext\'erieures}, J.
  Algebra \textbf{46} (1977), no.~1, 268--277. \MR{0472881 (57 \#12568)}

\bibitem[TW15]{TignolWadsworth:2015}
J.-P. Tignol and A.~R. Wadsworth, \emph{Value functions on simple algebras, and
  associated graded rings}, Springer Monographs in Mathematics, Springer, 2015.

\end{thebibliography}
\end{document}